\documentclass[10pt,a4paper,reqno]{amsart}

\usepackage[utf8]{inputenc}
\usepackage[T1]{fontenc}
\usepackage{lmodern}

\usepackage[english]{babel}

\usepackage{amsmath}
\usepackage{amsfonts}
\usepackage{amssymb}
\usepackage{amsthm}

\usepackage{hyperref}

\theoremstyle{plain}
	\newtheorem{theorem}{Theorem}
	\newtheorem{proposition}[theorem]{Proposition}
	\newtheorem{lemma}[theorem]{Lemma}
\theoremstyle{definition}
	\newtheorem{definition}[theorem]{Definition}
\theoremstyle{remark} 
	\newtheorem{remark}[theorem]{Remark}
	\newtheorem{example}[theorem]{Example}

\begin{document}

\title[Division gradings on finite-dimensional simple real algebras]{Classification of division gradings on finite-dimensional simple real algebras}
\author{Adri\'an Rodrigo-Escudero}
\date{December 22, 2015}
\address{Departamento de Matem\'aticas e Instituto Uni\-ver\-si\-ta\-rio de Matem\'aticas y Aplicaciones, Universidad de Zaragoza, 50009 Zaragoza, Spain.}
\email{adrian.rodrigo.escudero@gmail.com}
\thanks{\texttt{http://dx.doi.org/10.1016/j.laa.2015.11.025} This is the accepted manuscript (it includes the referee's suggestions) of the article published in Linear Algebra and its Applications.}

\begin{abstract}
We classify, up to isomorphism and up to equivalence, division gradings (by abelian groups) on finite-dimensional simple real algebras.
Gradings on finite-dimensional simple algebras are determined by division gradings, so our results give the classification, up to isomorphism, of (not necessarily division) gradings on such algebras.

Linear algebra over the field of two elements plays an interesting role in the proofs.
\end{abstract}

\maketitle

\section{Introduction}\label{sec-intr}

Consider a finite-dimensional simple real algebra (here by algebra we always mean associative algebra).
The Theorem of Artin--Wedderburn states that it is isomorphic to a matrix algebra over a finite-dimensional division real algebra.
On the other hand, Frobenius' Theorem characterizes the latter; it says that any finite-dimensional division real algebra is isomorphic to $\mathbb{R}$, $\mathbb{C}$ or $\mathbb{H}$.
Both theorems assure the uniqueness, so they classify these algebras.

A graded algebra is said to be a graded division algebra if it is unital and every nonzero homogeneous element has an inverse.
In \cite[Corollary 2.12]{EK} (see also \cite{BSZ,BZ02,NO}) an analogue of Artin--Wedderburn's Theorem for graded simple algebras is proved: they are, roughly speaking, graded matrix algebras over graded division algebras.
This gives a classification up to isomorphism of gradings.
The difficulty of studying the equivalence classes if the grading is not fine is illustrated in \cite[Examples 2.40 and 2.41]{EK}.
The classification will be completed if an analogue of Frobenius' Theorem for gradings is given, that is, if the graded division algebras are classified.
This is done in \cite[Theorem 2.15]{EK} for the case in which the ground field is algebraically closed; in this paper we will do it for the real case.
Therefore, the results in this text give a complete classification of the isomorphism classes of gradings on finite-dimensional simple real algebras.

The goal of this text is to classify, up to isomorphism and up to equivalence, division gradings (by abelian groups) on finite-dimensional simple real algebras.
This is achieved in Theorems \ref{thm-MnR-dim1}, \ref{thm-MnC-dim1}, \ref{thm-dim4}, \ref{thm-MnR-dim2} and \ref{thm-MnC-dim2}.
We advance in Section \ref{sec-equiv} a list of the equivalence classes.
The center of the algebra plays a key role; this is the reason to distinguish between the cases of $M_n(\mathbb{R})$ and $M_n(\mathbb{H})$ and the case of $M_n(\mathbb{C})$.

The homogeneous components of these gradings have the same dimension, which is restricted to $1$, $2$ or $4$ (see Section \ref{sec-back-grad}).
The former case is analyzed in Section \ref{sec-fine}.
The situation is very similar to that of the complex field, but the isomorphism classes are in correspondence with quadratic forms over the field of two elements, instead of alternating bicharacters, because of the lack of roots of unity in the real numbers.
Section \ref{sec-quad-form} is devoted to quadratic forms.
As a curiosity, we obtain an alternative proof that the Arf invariant over the field of two elements is well defined, working with algebras in characteristic $0$.

Finally in Section \ref{sec-coar} we study the remaining case of homogeneous components of dimension $2$ or $4$.
We have two tools to reduce the problem to the previous case.
The first is Proposition \ref{prop-ref}, which states that these gradings are not fine.
The second is the Double Centralizer Theorem.

In the recent preprint \cite{BZ15} (simultaneous to our preprint \cite{Rodrigo}), a classification up to equivalence of the division gradings on finite-dimensional simple real algebras has been considered.
We solve, using different methods, the problem of classification both up to equivalence and up to isomorphism.

\section{Background on gradings}\label{sec-back-grad}

In this section we review, following \cite{EK}, the basic definitions and properties of gradings that will be used in the rest of the paper.

\begin{definition}\label{def-grad}
Let $\mathcal{D}$ be an algebra over a field $\mathbb{F}$, and let $G$ be a group.
A \emph{$G$-grading} $\Gamma$ on $\mathcal{D}$ is a decomposition of $\mathcal{D}$ into a direct sum of subspaces indexed by $G$:
\[ \Gamma : \mathcal{D} = \bigoplus_{ g \in G } \mathcal{D}_g \]
So that for all $ g,h \in G$:
\[ \mathcal{D}_g \mathcal{D}_h \subseteq \mathcal{D}_{gh} \]
If such a decomposition is fixed, we refer to $\mathcal{D}$ as a \emph{$G$-graded algebra}.
The \emph{support} of $\Gamma$ is the set $ \mathrm{supp} (\Gamma) := \{ g \in G \mid \mathcal{D}_g \neq 0 \} $.
If $ X \in \mathcal{D}_g $, then we say that $X$ is \emph{homogeneous of degree $g$}, and we write $ \deg (X) = g $.
The subspace $\mathcal{D}_g$ is called the \emph{homogeneous component of degree $g$}.
\end{definition}

Note that, if $\mathcal{D}$ is a $G$-graded algebra and $\mathcal{D}'$ is an $H$-graded algebra, then the tensor product $ \mathcal{D} \otimes \mathcal{D}' $ has a natural $ G \times H $-grading given by $ ( \mathcal{D} \otimes \mathcal{D}' )_{(g,h)} = \mathcal{D}_g \otimes \mathcal{D}'_h $, for all $ g \in G $, $ h \in H $.
This will be called the \emph{product grading}.

A subalgebra $\mathcal{F}$ of a $G$-graded algebra $\mathcal{D}$ is said to be \emph{graded} if $ \mathcal{F} = \bigoplus_{ g \in G } ( \mathcal{D}_g \cap \mathcal{F} ) $.

There are two natural ways to define an equivalence relation on group gradings, depending on whether the grading group plays a secondary role or it is a part of the definition.

\begin{definition}\label{def-equiv}
Let $\Gamma$ be a $G$-grading on the algebra $\mathcal{D}$ and let $\Gamma'$ be an $H$-grading on the algebra $\mathcal{D}'$.
We say that $\Gamma$ and $\Gamma'$ are \emph{equivalent} if there exist an isomorphism of algebras $ \varphi : \mathcal{D} \longrightarrow \mathcal{D}' $ and a bijection $ \alpha : \mathrm{supp} (\Gamma) \longrightarrow \mathrm{supp} (\Gamma') $ such that $ \varphi (\mathcal{D}_t) = \mathcal{D}'_{\alpha(t)} $ for all $ t \in \mathrm{supp} (\Gamma) $.
\end{definition}

\begin{definition}\label{def-isom}
Let $\Gamma$ and $\Gamma'$ be $G$-gradings on the algebras $\mathcal{D}$ and $\mathcal{D}'$, respectively.
We say that $\Gamma$ and $\Gamma'$ are \emph{isomorphic} if there exists an isomorphism of algebras $ \varphi : \mathcal{D} \longrightarrow \mathcal{D}' $ such that $ \varphi (\mathcal{D}_g) = \mathcal{D}'_g $ for all $ g \in G $.
\end{definition}

\begin{definition}\label{def-ref-coars}
Given gradings $ \Gamma : \mathcal{D} = \bigoplus_{ g \in G } \mathcal{D}_g $ and $ \Gamma' : \mathcal{D}' = \bigoplus_{ h \in H } \mathcal{D}'_h $, we say that $\Gamma'$ is a \emph{coarsening} of $\Gamma$, or that $\Gamma$ is a \emph{refinement} of $\Gamma'$, if for any $ g \in G $ there exists $ h \in H $ such that $ \mathcal{D}_g \subseteq \mathcal{D}'_h $.
If for some $ g \in G $ this inclusion is strict, then we will speak of a \emph{proper} refinement or coarsening.
A grading is said to be \emph{fine} if it does not admit a proper refinement.
\end{definition}

\begin{definition}\label{def-grad-div}
A graded algebra is said to be a \emph{graded division algebra} if it is unital and every nonzero homogeneous element has an inverse.
\end{definition}

If $\mathcal{D}$ is a $G$-graded division algebra, then $ I \in \mathcal{D}_e $, where $e$ is the neutral element of $G$ and $I$ the unity of $\mathcal{D}$.
Also, if $ 0 \neq X \in \mathcal{D}_g $, then $ X^{-1} \in \mathcal{D}_{g^{-1}} $.
Therefore, the support of $\mathcal{D}$ is a subgroup of $G$, since whenever $ \mathcal{D}_g \neq 0 $ and $ \mathcal{D}_h \neq 0 $, we also have $ 0 \neq \mathcal{D}_g \mathcal{D}_h \subseteq \mathcal{D}_{gh} $ and $ \mathcal{D}_{g^{-1}} \neq 0 $.
This also shows that, in the situation of Definition \ref{def-equiv}, if $\Gamma$ and $\Gamma'$ are division gradings, then $ \alpha : \mathrm{supp} (\Gamma) \longrightarrow \mathrm{supp} (\Gamma') $ is a homomorphism of groups.

\begin{remark}\label{rem-dim-hom-comp}
The neutral component $ \mathcal{D}_e $ of a graded division algebra $\mathcal{D}$ is a division algebra.
Also, if $ X_t \in \mathcal{D}_t $ is nonzero, then $ \mathcal{D}_t = \mathcal{D}_e X_t $; therefore all the (nonzero) homogeneous components of the grading have the same dimension.
In our case $\mathcal{D}$ will be finite-dimensional and the ground field will be the real field $\mathbb{R}$,
so this dimension will be 1, 2 or 4 depending on whether $ \mathcal{D}_e $ is isomorphic to $ \mathbb{R} $, $ \mathbb{C} $ or $ \mathbb{H} $.
\end{remark}

\section{Equivalence classes}\label{sec-equiv}

In this section we present a list of representatives of the equivalence classes of all division gradings (by abelian groups) on finite-dimensional simple real algebras, that is, on $M_n(\mathbb{R})$, $M_n(\mathbb{H})$ and $M_n(\mathbb{C})$.
Such gradings only exist if $n=2^m$ for some integer $ m \geq 0 $ (with the exception of $M_n(\mathbb{C})$ and its gradings as a complex algebra).
The proof that it is an exhaustive list with no redundancy of the equivalence classes, together with the classification up to isomorphism, will be given in Sections \ref{sec-fine} and \ref{sec-coar}.
We will need the following building blocks.

\begin{example}\label{exam-H-M2R-dim1}
We have the following $\mathbb{Z}_2^2$-division grading on $\mathbb{H}$:
\begin{equation}\label{eq-H-dim1}
\mathbb{H} = \mathbb{R} 1 \oplus \mathbb{R} i \oplus \mathbb{R} j \oplus \mathbb{R} k
\end{equation}
And on $ M_2 ( \mathbb{R} ) $ (also with the grading group $\mathbb{Z}_2^2$):
\begin{equation}\label{eq-M2R-dim1}
M_2(\mathbb{R}) =
	\mathbb{R} \begin{pmatrix} 1 & 0 \\ 0 & 1 \end{pmatrix}
	\oplus
	\mathbb{R} \begin{pmatrix} 0 & 1 \\ 1 & 0 \end{pmatrix}
	\oplus
	\mathbb{R} \begin{pmatrix} -1 & 0 \\ 0 & 1 \end{pmatrix}
	\oplus
	\mathbb{R} \begin{pmatrix} 0 & -1 \\ 1 & 0 \end{pmatrix}
\end{equation}
$\mathbb{C}$ has a natural $\mathbb{Z}_2$-division grading:
\begin{equation}\label{eq-C-dim1}
\mathbb{C} = \mathbb{R} 1 \oplus \mathbb{R} i
\end{equation}
Also, we can define the following division grading on $ M_2 ( \mathbb{C} ) $ by the group $ \mathbb{Z}_2 \times \mathbb{Z}_4 = \{ e,b,b^2,b^3,a,ab,ab^2,ab^3 \} $:
\begin{align}\label{eq-M2C-dim1}
M_2(\mathbb{C}) = &
	\mathbb{R} \begin{pmatrix} 1 & 0 \\ 0 & 1 \end{pmatrix}
	\oplus
	\mathbb{R} \begin{pmatrix} \omega & 0 \\ 0 & -\omega \end{pmatrix}
	\oplus
	\mathbb{R} \begin{pmatrix} i & 0 \\ 0 & i \end{pmatrix}
	\oplus
	\mathbb{R} \begin{pmatrix} \omega^3 & 0 \\ 0 & -\omega^3 \end{pmatrix}
	\oplus
\nonumber \\
&
	\mathbb{R} \begin{pmatrix} 0 & 1 \\ 1 & 0 \end{pmatrix}
	\oplus
	\mathbb{R} \begin{pmatrix} 0 & -\omega \\ \omega & 0 \end{pmatrix}
	\oplus
	\mathbb{R} \begin{pmatrix} 0 & i \\ i & 0 \end{pmatrix}
	\oplus
	\mathbb{R} \begin{pmatrix} 0 & -\omega^3 \\ \omega^3 & 0 \end{pmatrix}
\end{align}
Where $ \omega := e^{ i \pi / 4 } $ (note that $ \{ (1,1) , (\omega,-\omega) \} $ is a $\mathbb{C}$-basis of $\mathbb{C}^2$, therefore $ \{ (1,1) , \allowbreak (\omega,-\omega) , \allowbreak (i,i) , \allowbreak (\omega^3,-\omega^3) \} $ is an $\mathbb{R}$-basis of $\mathbb{C}^2$).
\end{example}

\begin{example}\label{exam-H-M2R-M2C-dim2}
We have the following division gradings on $\mathbb{H}$, $M_2(\mathbb{R})$ and $M_2(\mathbb{C})$ with supports $\mathbb{Z}_2$, $\mathbb{Z}_2$ and $\mathbb{Z}_4$, respectively (they are obtained as coarsenings of \eqref{eq-H-dim1}, \eqref{eq-M2R-dim1} and \eqref{eq-M2C-dim1}, respectively):
\begin{equation}\label{eq-H-dim2}
\mathbb{H} = [ \mathbb{R} 1 \oplus \mathbb{R} i ] \oplus [ \mathbb{R} j \oplus \mathbb{R} k ]
\end{equation}
\begin{equation}\label{eq-M2R-dim2}
M_2(\mathbb{R}) =
\left[
	\mathbb{R} \begin{pmatrix} 1 & 0 \\ 0 & 1 \end{pmatrix}
	\oplus
	\mathbb{R} \begin{pmatrix} 0 & -1 \\ 1 & 0 \end{pmatrix}
\right]
	\oplus
\left[
	\mathbb{R} \begin{pmatrix} 0 & 1 \\ 1 & 0 \end{pmatrix}
	\oplus
	\mathbb{R} \begin{pmatrix} -1 & 0 \\ 0 & 1 \end{pmatrix}
\right]
\end{equation}
\begin{align}\label{eq-M2C-dim2}
M_2(\mathbb{C}) = &
\left[
	\mathbb{R} \begin{pmatrix} 1 & 0 \\ 0 & 1 \end{pmatrix}
	\oplus
	\mathbb{R} \begin{pmatrix} 0 & i \\ i & 0 \end{pmatrix}
\right]
	\oplus
\left[
	\mathbb{R} \begin{pmatrix} \omega & 0 \\ 0 & -\omega \end{pmatrix}
	\oplus
	\mathbb{R} \begin{pmatrix} 0 & -\omega^3 \\ \omega^3 & 0 \end{pmatrix}
\right]
	\oplus
\nonumber \\
&
\left[
	\mathbb{R} \begin{pmatrix} i & 0 \\ 0 & i \end{pmatrix}
	\oplus
	\mathbb{R} \begin{pmatrix} 0 & 1 \\ 1 & 0 \end{pmatrix}
\right]
	\oplus
\left[
	\mathbb{R} \begin{pmatrix} \omega^3 & 0 \\ 0 & -\omega^3 \end{pmatrix}
	\oplus
	\mathbb{R} \begin{pmatrix} 0 & -\omega \\ \omega & 0 \end{pmatrix}
\right]
\end{align}
Where $ \omega := e^{ i \pi / 4 } $.
They are determined by the fact that there exists an $ X \in \mathcal{D}_{\bar{1}} $ such that $XJ=-JX$ and either $X^2=-I$, $X^2=+I$ or $X^4=-I$, respectively (where $I$ is the identity matrix and $ J \in \mathcal{D}_{\bar{0}} $ satisfies $J^2=-I$).
\end{example}

Now we give the promised list of representatives of the equivalence classes of division gradings.
In all cases, with the exception of (2).(f), $n=2^m$.
The gradings considered in the tensor products are the product gradings.

\begin{enumerate}
\item \textbf{Homogeneous components of dimension 1}.
Here all the copies of $M_2(\mathbb{R})$, $\mathbb{H}$, $\mathbb{C}$ and $M_2(\mathbb{C})$ are graded as in Example \ref{exam-H-M2R-dim1}.
\begin{enumerate}
	\item $ M_n(\mathbb{R}) \cong M_2(\mathbb{R}) \otimes \ldots \otimes M_2(\mathbb{R}) $.
	The grading group is $ (\mathbb{Z}_2^2)^m \cong \mathbb{Z}_2^{2m} $.
	If $n=1$ (that is, $m=0$) $ M_n(\mathbb{R}) = \mathbb{R} $ with the trivial grading.
	\item $ M_n(\mathbb{H}) \cong M_2(\mathbb{R}) \otimes \ldots \otimes M_2(\mathbb{R}) \otimes \mathbb{H} $.
	The grading group is $ (\mathbb{Z}_2^2)^{m+1} \cong \mathbb{Z}_2^{2m+2} $.
	\item $ M_n(\mathbb{C}) \cong M_2(\mathbb{R}) \otimes \ldots \otimes M_2(\mathbb{R}) \otimes \mathbb{C} $.
	The grading group is $ (\mathbb{Z}_2^2)^m \times \mathbb{Z}_2 \cong \mathbb{Z}_2^{2m+1} $.
	\item $ M_n(\mathbb{C}) \cong M_2(\mathbb{R}) \otimes \ldots \otimes M_2(\mathbb{R}) \otimes M_2(\mathbb{C}) $.
	The grading group is $ (\mathbb{Z}_2^2)^{m-1} \times ( \mathbb{Z}_2 \times \mathbb{Z}_4 ) \cong \mathbb{Z}_2^{2m-1} \times \mathbb{Z}_4 $, and $ n \geq 2 $.
\end{enumerate}
\item \textbf{Homogeneous components of dimension 2}.
Here the first factor in each tensor product is graded as in Example \ref{exam-H-M2R-M2C-dim2}; the other factors are graded as in Example \ref{exam-H-M2R-dim1}.
\begin{enumerate}
	\item $ M_n(\mathbb{R}) \cong M_2(\mathbb{R}) \otimes M_2(\mathbb{R}) \otimes \ldots \otimes M_2(\mathbb{R}) $.
	The grading group is $ \mathbb{Z}_2 \times (\mathbb{Z}_2^2)^{m-1} \cong \mathbb{Z}_2^{2m-1} $, and $ n \geq 2 $.
	\item $ M_n(\mathbb{H}) \cong \mathbb{H} \otimes M_2(\mathbb{R}) \otimes \ldots \otimes M_2(\mathbb{R}) $.
	The grading group is $ \mathbb{Z}_2 \times (\mathbb{Z}_2^2)^m \cong \mathbb{Z}_2^{2m+1} $.
	\item $ M_n(\mathbb{C}) \cong M_2(\mathbb{R}) \otimes M_2(\mathbb{R}) \otimes \ldots \otimes M_2(\mathbb{R}) \otimes \mathbb{C} $.
	The grading group is $ \mathbb{Z}_2 \times (\mathbb{Z}_2^2)^{m-1} \times \mathbb{Z}_2 \cong \mathbb{Z}_2^{2m} $, and $ n \geq 2 $.
	\item $ M_n(\mathbb{C}) \cong M_2(\mathbb{R}) \otimes M_2(\mathbb{R}) \otimes \ldots \otimes M_2(\mathbb{R}) \otimes M_2(\mathbb{C}) $.
	The grading group is $ \mathbb{Z}_2 \times (\mathbb{Z}_2^2)^{m-2} \times ( \mathbb{Z}_2 \times \mathbb{Z}_4 ) \cong \mathbb{Z}_2^{2m-2} \times \mathbb{Z}_4 $, and $ n \geq 4 $.
	\item $ M_n(\mathbb{C}) \cong M_2(\mathbb{C}) \otimes M_2(\mathbb{R}) \otimes \ldots \otimes M_2(\mathbb{R}) $.
	The grading group is $ \mathbb{Z}_4 \times (\mathbb{Z}_2^2)^{m-1} \cong \mathbb{Z}_2^{2m-2} \times \mathbb{Z}_4 $, and $ n \geq 2 $.
	\item Gradings on $M_n(\mathbb{C})$ as a complex algebra (see Remark \ref{rem-complex-grad}).
\end{enumerate}
\item \textbf{Homogeneous components of dimension 4}.
Here the first factor in each tensor product (which is always $\mathbb{H}$) is graded by the trivial group; the other factors are graded as in Example \ref{exam-H-M2R-dim1}.
\begin{enumerate}
	\item $ M_n(\mathbb{R}) \cong \mathbb{H} \otimes M_2(\mathbb{R}) \otimes \ldots \otimes M_2(\mathbb{R}) \otimes \mathbb{H} $.
	The grading group is $ (\mathbb{Z}_2^2)^{m-1} \cong \mathbb{Z}_2^{2m-2} $, and $ n \geq 4 $.
	\item $ M_n(\mathbb{H}) \cong \mathbb{H} \otimes M_2(\mathbb{R}) \otimes \ldots \otimes M_2(\mathbb{R}) $.
	The grading group is $ (\mathbb{Z}_2^2)^{m} \cong \mathbb{Z}_2^{2m} $.
	\item $ M_n(\mathbb{C}) \cong \mathbb{H} \otimes M_2(\mathbb{R}) \otimes \ldots \otimes M_2(\mathbb{R}) \otimes \mathbb{C} $.
	The grading group is $ (\mathbb{Z}_2^2)^{m-1} \times \mathbb{Z}_2 \cong \mathbb{Z}_2^{2m-1} $, and $ n \geq 2 $.
	\item $ M_n(\mathbb{C}) \cong \mathbb{H} \otimes M_2(\mathbb{R}) \otimes \ldots \otimes M_2(\mathbb{R}) \otimes M_2(\mathbb{C}) $.
	The grading group is $ (\mathbb{Z}_2^2)^{m-2} \times ( \mathbb{Z}_2 \times \mathbb{Z}_4 ) \cong \mathbb{Z}_2^{2m-3} \times \mathbb{Z}_4 $, and $ n \geq 4 $.
\end{enumerate}
\end{enumerate}

Since the dimension of the homogeneous components and the universal grading group are invariants of the equivalence classes of gradings, it is clear that the gradings in the above list are pairwise nonequivalent, with the possible exceptions of the pair (2).(c) and (2).(f), and the pair (2).(d) and (2).(e), which are studied in Section \ref{sec-coar}.

\section{Quadratic forms}\label{sec-quad-form}

We will classify the isomorphism classes of division gradings in terms of the maps that we present in this section.
If the homogeneous components have dimension $1$ or $4$, the isomorphism classes are in correspondence with quadratic forms over the field of two elements.

Let $T$ be an abelian group, an \emph{($\mathbb{R}$-valued) alternating bicharacter} is a map $ \beta : T \times T \longrightarrow \mathbb{R}^{\times} $ which is multiplicative in each variable and satisfies $ \beta (t,t) = 1 $ for all $ t \in T $.

Given an elementary abelian $2$-group $ T \cong \mathbb{Z}_2 \times \ldots \times \mathbb{Z}_2 $, we may think of it as a vector space over the field $ \mathbb{F}_2 $ ($ = \mathbb{Z}_2 $).
Recall that a \emph{quadratic form} on $T$ is a map $ \mu : T \longrightarrow \mathbb{F}_2 $ such that $ \mu ( \alpha t ) = \alpha^2 \mu (t) $ for all $ \alpha \in \mathbb{F}_2 $, $ t \in T $ and that the map $ \beta : T \times T \longrightarrow \mathbb{F}_2 $ defined by $ \beta (u,v) := \mu(u+v) - \mu(u) - \mu(v) $ is bilinear.
Note that the first condition is equivalent to $\mu(0)=0$; and that this implies that $ \beta $ is alternating, that is, $ \beta(t,t)=0 $ for all $ t \in T $.
The quadratic form $ \mu $ is \emph{regular} if one of the following conditions holds: either $ \mathrm{rad} (\beta) = \{ t \in T \mid \beta(u,t)=0 , \, \forall u \in T \} $ is zero, or $ \dim ( \mathrm{rad} (\beta) ) = 1 $ and $ \mu ( \mathrm{rad} (\beta) ) \neq 0 $.
It will be convenient for us to say that $ \mu $ has \emph{type I} in the first case and \emph{type II} in the second, and to call the generator of $ \mathrm{rad} (\beta) $ (for type II) the \emph{semineutral element} of the group.

Actually, we will use multiplicative versions, so in what follows the relation between a quadratic form $ \mu : T \longrightarrow \{ \pm 1 \} $ and its associated alternating bicharacter $ \beta : T \times T \longrightarrow \{ \pm 1 \} $ will be given by:
\[ \beta(u,v) = \mu(uv) \mu(u)^{-1} \mu(v)^{-1} \]
Note that $\mu(e)=+1$ is equivalent to $\beta$ being alternating.
Remark that $\mu$ has type I if $ \mathrm{rad} (\beta) = \{ t \in T \mid \beta(u,t)=+1 , \, \forall u \in T \} = \{ e \} $; while it has type II if there exists a semineutral element $ f \in T $ such that $ \mu (f) =-1 $ and $ \mathrm{rad} (\beta) = \{ e,f \} $.

Given an alternating bicharacter $ \beta : T \times T \longrightarrow \mathbb{R}^{\times} $, where the abelian group $T$ is not necessarily an elementary abelian $2$-group, we will also say that it has \emph{type I} or \emph{II} (and \emph{semineutral element}) if its radical satisfies the previous conditions ($ \mathrm{rad} (\beta) = \{ e \} $ for type I, and $ \mathrm{rad} (\beta) = \{ e,f \} $ for type II).
Note that if $T$ is finite, $\beta(u,v)$ is a root of unity for any $ u,v \in T $, so $\beta(u,v)$ is either $+1$ or $-1$.

\begin{proposition}\label{prop-basis-alt-bich}
Let $\beta$ be an alternating bicharacter on a finite abelian group $T$.
We say that a family $ \{ a_1 , b_1 , \ldots , a_m , b_m \} \subseteq T $ is symplectic if $ \beta (a_i,b_i) = \beta (b_i,a_i) = -1 $ ($ i = 1 , \ldots , m $) and all the other $\beta$-products are $+1$.
\begin{enumerate}
	\item If $\beta$ has type I, then $ T \cong (\mathbb{Z}_2^2)^m $ and there exists a symplectic basis $ \{ a_1 , \allowbreak b_1 , \ldots , a_m , \allowbreak b_m \} $ of $T$ (as a $\mathbb{F}_2$-vector space).
	\item If $\beta$ has type II with semineutral element $f$, then either $ T \cong (\mathbb{Z}_2^2)^m \times \mathbb{Z}_2 $ or $ (\mathbb{Z}_2^2)^{m-1} \times ( \mathbb{Z}_2 \times \mathbb{Z}_4 ) $, and there is a symplectic family $ \{ a_1 , \allowbreak b_1 , \ldots , a_m , \allowbreak b_m \} $ such that:
	\begin{enumerate}
		\item In the first case, $ \{ a_1 , b_1 , \ldots , a_m , b_m , f \} $ is a basis of $T$ (as a $\mathbb{F}_2$-vector space).
		\item In the second case, $ b_m^2 = f $ and $T$ is the inner direct product of the subgroups $ \langle a_1 \rangle , \langle b_1 \rangle , \ldots , \langle a_m \rangle , \langle b_m \rangle$ ($a_i^2=e$ for all $ i = 1 , \ldots , m $, and $b_i^2=e$ for all $ i = 1 , \ldots , m-1 $).
	\end{enumerate}
\end{enumerate}
\end{proposition}

\begin{proof}
Part 1 is known, see for example \cite[Equation (2.6) in page 36]{EK}.
For part 2 we consider the alternating bicharacter $ \tilde{\beta} $ on $ T / \langle f \rangle $ such that $ \beta = \tilde{\beta} \circ ( \pi \times \pi ) $, where $ \pi : T \longrightarrow T / \langle f \rangle $ is the natural projection.
$ \tilde{\beta} $ satisfies part 1, so $ T / \langle f \rangle \cong \mathbb{Z}_2^{2m} $, and necessarily $ T \cong \mathbb{Z}_2^{2m+1} $ or $ \mathbb{Z}_2^{2m-1} \times \mathbb{Z}_4 $.
Now for (a) we can take any $ a_1 , b_1 , \ldots , a_m , b_m $ such that $ \{ a_1 \langle f \rangle , b_1 \langle f \rangle , \ldots , a_m \langle f \rangle , b_m \langle f \rangle \} $ is a symplectic basis of $ T / \langle f \rangle $.

So assume that $ T \cong \mathbb{Z}_2^{2m-1} \times \mathbb{Z}_4 $.
We define $S$ as the set of those $ u \in T \setminus \{f\} $ of order 2 such that for all $ v \in T $ of order 2 we have $\beta(u,v)=+1$.
If $ u \in S $, then $\beta(u,w)=-1$ for all $ w \in T $ of order 4, because $ u \not \in \mathrm{rad}(\beta) $ and the difference of two elements of $T$ of order 4 is an element of order 2.
So $S$ has at most two elements.

Now we prove (b) by induction on $m$.
If $m=1$ the statement is clear.
While if $m>1$, as $ \vert S \vert \leq 2 $, there exist $ a_1,b_1 \in T $ of order 2 such that $\beta(a_1,b_1)=-1$.
We can decompose $T$ as $ H \times H^{\perp} $, where $ H := \langle a_1 , b_1 \rangle $ and $ H^{\perp} := \{ t \in T \mid \beta(a_1,t) = \beta(b_1,t) = +1 \} $, and apply the induction hypothesis.
\end{proof}

\begin{remark}\label{rem-basis-alt-bich}
Let $ T \cong \mathbb{Z}_2^{2m-1} \times \mathbb{Z}_4 $, as in part 2.(b) of the previous proposition; and denote $ T^2 = \{ t^2 \mid t \in T \} $ and $ T_2 := \{ t \in T \mid t^2=e \} $.
Since the subgroup $T^2$ has only two elements, its generator $f$ is a distinguished element of the group.
If we consider an alternating bicharacter $\beta$ on $T$ of type II, the semineutral element will be $f$.

The presence of $\beta$ also defines two more distinguished elements, the members of $ \mathrm{rad} ( \beta \vert_{ T_2 \times T_2 } ) \setminus \mathrm{rad}(\beta) $.
Indeed, with the notation of the previous proposition, they are $a_m$ and $ a_m f $.
Moreover, if the restriction to $ T_2 \times T_2 $ of $\beta$ is associated to a quadratic form $\mu$ such that $\mu(f)=-1$, then $\mu$ takes the value $+1$ on one of them and $-1$ on the other; so they are determined completely by the pair $(\beta,\mu)$.
\end{remark}

\begin{remark}[Arf invariant]\label{rem-arf-inv}
Let $\mu$ be a quadratic form of type I and consider a symplectic basis $ \{ a_1 , b_1 , \ldots , a_m , b_m \} $ of its associated alternating bicharacter as in Proposition \ref{prop-basis-alt-bich}.
Denote by $ m : \{ \pm 1 \} \times \{ \pm 1 \} \longrightarrow \{ \pm 1 \} $ the multiplicative version of the product in $\mathbb{Z}_2$, that is, $ m(+1,+1) = m(+1,-1) = m(-1,+1) = +1 $ and $ m(-1,-1) = -1 $.
The \emph{Arf invariant} of $\mu$ is:
\[ \mathrm{Arf}(\mu) := m(\mu(a_1),\mu(b_1)) \cdots m(\mu(a_m),\mu(b_m)) \]
It is well known (see for example \cite[Corollary 9.5]{Saveliev}) that $\mathrm{Arf}(\mu)$ is the value which is assumed most often by $\mu$; therefore, the definition does not depend on the choice of the symplectic basis.
Actually, the proof of Theorem \ref{thm-MnR-dim1} below provides an alternative proof of the fact that the Arf invariant is well defined.
\end{remark}

If the dimension of the homogeneous components is $2$, the situation is more complicated.
We will now define the objects that are in correspondence with the isomorphism classes of gradings.

\begin{lemma}\label{lem-nice}
Let $T$ be a finite abelian group, $K$ an index $2$ subgroup of $T$, and $ \nu : T \setminus K \longrightarrow \{ \pm 1 \} $ a map.
Suppose that $K$ is an elementary abelian $2$-group (that is, $ K \cong \mathbb{Z}_2 \times \ldots \times \mathbb{Z}_2 $).
Consider the family of maps $ \mu_g : K \longrightarrow \{ \pm 1 \} $ defined by $ \mu_g(k) := \nu(gk) \nu(g)^{-1} $, as $g$ runs through $ T \setminus K $.
Then, if a member of this family is a quadratic form, so are the others, and all have the same associated alternating bicharacter.
Moreover, all the quadratic forms are regular if one of them is.
Besides, if the alternating bicharacter has type I, then the product $ \nu(g) \cdot \mathrm{Arf} (\mu_g) $ does not depend on the choice of $ g \in T \setminus K $.
\end{lemma}

As it happens with the Arf invariant, this lemma follows from the proof of Theorems \ref{thm-MnR-dim2} and \ref{thm-MnC-dim2}.
Anyway, we give a direct proof.

\begin{proof}
Let $ g,h \in T \setminus K $, and assume that $ \mu_g $ is a quadratic form.
Call $ \beta $ its associated alternating bicharacter.
The first assertions follow from the following formula ($ k \in K $):
\[ \mu_h (k) = \frac{\nu(hk)}{\nu(h)} = \frac{\mu_g(g^{-1}hk)}{\mu_g(g^{-1}h)} = \mu_g (k) \beta (g^{-1}h,k) \]
Indeed, as $\beta$ is bimultiplicative, $ \mu_h $ is also a quadratic form with the same $\beta$.
While if $ \mu_g $ has type II with semineutral element $f$, then $ \mu_g (f) = -1 $, so also $ \mu_h (f) = \mu_g (f) \beta (g^{-1}h,f) = -1 $.

Finally suppose that the quadratic forms have type I.
$ \mathrm{Arf} (\mu_g) $ is the value which is assumed most often by $ \mu_g $.
And $ \nu(gk) = \nu(g) \mu_g(k) $ for all $ k \in K $.
Therefore $ \nu(g) \cdot \mathrm{Arf} (\mu_g) $ takes the value which is assumed most often by $\nu$, which is independent of $ g \in T \setminus K $.
\end{proof}

\begin{definition}\label{def-nice-map}
With the notation of the previous lemma, we will say that $\nu$ is a \emph{nice map} if the $\mu_g$ are quadratic forms; that it is \emph{regular} or that it is of \emph{type I} or \emph{II} if so are the $\mu_g$; and, in the case of type I, that it is \emph{positive} or \emph{negative} if the value of $ \nu(g) \cdot \mathrm{Arf} (\mu_g) $ is $+1$ or $-1$, respectively.
Also, we will speak of the \emph{semineutral element} of $\nu$ (when $\nu$ has type II) referring to that of the quadratic forms $\mu_g$.
\end{definition}

\section{Fine gradings}\label{sec-fine}

In this section we classify the division gradings whose homogeneous components have dimension $1$ (see Remark \ref{rem-dim-hom-comp}).
As mentioned in the Introduction, the situation is very similar to that of the complex field, but the isomorphism classes are in correspondence with quadratic forms over the field of two elements, instead of alternating bicharacters, because of the lack of roots of unity in the real numbers.
Later on, we will see in Proposition \ref{prop-ref} that these are precisely the fine gradings.

The center of the algebra will play a key role in the following arguments.
It will force us to differentiate between the case of $M_n(\mathbb{R})$ and $M_n(\mathbb{H})$ and the case of $M_n(\mathbb{C})$.

\begin{lemma}\label{lem-iI}
Consider a grading on a real algebra $\mathcal{D}$ isomorphic to $M_n (\mathbb{C})$ by an abelian group $G$.
Denote by $I$ the unity of $\mathcal{D}$.
Then the element $iI$ is homogeneous of degree either the neutral element $e$ or an element $f$ of order 2.
\end{lemma}

\begin{proof}
As $G$ is abelian, the center $Z(\mathcal{D})$ is a graded subalgebra, because $ x = \sum_{u \in G} x_u \in Z(\mathcal{D}) $ if and only if $ x y_v = y_v x $ for all $ v \in G$, $ y_v \in \mathcal{D}_v $, if and only if $ x_u y_v = y_v x_u $ for all $ u \in G$, $ v \in G$, $ y_v \in \mathcal{D}_v $, if and only if $ x_u \in Z(\mathcal{D}) $ for all $ u \in G$.
If $Z(\mathcal{D})$ is trivially graded we are done, so assume that $ iI = a + b $ with $ a \in Z(\mathcal{D})_e = \mathbb{R} I $ and $ b \in Z(\mathcal{D})_f $ ($ e \neq f $).
Necessarily $f$ has order 2.
Therefore considering the homogeneous components in $ -I = (iI)^2 = a^2 + b^2 + 2ab $ we get that $2ab=0$, so $a=0$.
\end{proof}

The element $f$ of the previous lemma will be the semineutral element of an alternating bicharacter.

\begin{theorem}\label{thm-MnR-dim1}
Let $G$ be an abelian group, and $\mathcal{D}$ a real algebra isomorphic to $ M_n (\mathbb{R}) $ (respectively $ M_n (\mathbb{H}) $).
If there is a division $G$-grading $\Gamma$ on $\mathcal{D}$ with homogeneous components of dimension $1$, then $ n = 2^m $ for some integer $ m \geq 0 $ and the support of $\Gamma$ is a subgroup $T$ of $G$ isomorphic to $\mathbb{Z}_2^{2m}$ (respectively $\mathbb{Z}_2^{2m+2}$).
The isomorphism classes of such gradings are in one-to-one correspondence with the pairs $(T,\mu)$ formed by a subgroup $T$ of $G$ isomorphic to $\mathbb{Z}_2^{2m}$ (respectively $\mathbb{Z}_2^{2m+2}$), and a quadratic form $\mu$ on $T$ of type I and Arf invariant $+1$ (respectively $-1$).
All such gradings belong to one equivalence class, represented by 1.(a) (respectively 1.(b)) in the list of Section \ref{sec-equiv}.
\end{theorem}

\begin{theorem}\label{thm-MnC-dim1}
Let $G$ be an abelian group, and $\mathcal{D}$ a real algebra isomorphic to $M_n(\mathbb{C})$.
If there is a division $G$-grading $\Gamma$ on $\mathcal{D}$ with homogeneous components of dimension $1$, then $ n = 2^m $ for some integer $ m \geq 0 $ and the support of $\Gamma$ is a subgroup $T$ of $G$ isomorphic either to $ \mathbb{Z}_2^{2m+1} $ or to $ \mathbb{Z}_2^{2m-1} \times \mathbb{Z}_4 $.
Moreover, the following possibilities appear:
\begin{enumerate}
	\item If the support is isomorphic to $\mathbb{Z}_2^{2m+1}$, the isomorphism classes of such gradings are in one-to-one correspondence with the pairs $(T,\mu)$ formed by a subgroup $T$ of $G$ isomorphic to $\mathbb{Z}_2^{2m+1}$, and a quadratic form $\mu$ on $T$ of type II.
	All such gradings belong to one equivalence class, represented by 1.(c) in the list of Section \ref{sec-equiv}.
	\item If the support is isomorphic to $ \mathbb{Z}_2^{2m-1} \times \mathbb{Z}_4 $, the isomorphism classes of such gradings are in one-to-one correspondence with the triples $(T,\beta,\mu)$ formed by a subgroup $T$ of $G$ isomorphic to $ \mathbb{Z}_2^{2m-1} \times \mathbb{Z}_4 $, an alternating bicharacter $\beta$ on $T$ of type II, and a quadratic form $\mu$ on $ T_2 := \{ t \in T \mid t^2=e \} $; such that the restriction of $\beta$ to $ T_2 \times T_2 $ is the associated alternating bicharacter of $\mu$, and that $\mu$ takes the value $-1$ on $f$, where $f$ is the generator of $ T^2 = \{ t^2 \mid t \in T \} $.
	All such gradings belong to one equivalence class, represented by 1.(d) in the list of Section \ref{sec-equiv}.
\end{enumerate}
\end{theorem}

\begin{proof}[Proof of Theorems \ref{thm-MnR-dim1} and \ref{thm-MnC-dim1}]
Let $T$ be the support of the grading.
For each $ t \in T $, we pick $ X_t \in \mathcal{D}_t $ so that $X_e$ is the unity $I$ and $ X_t^{ \vert t \vert } = \pm I $.
We also define $ \beta : T \times T \longrightarrow \mathbb{R}^{\times} $ by the equation:
\begin{equation}\label{eq-alt-bich}
X_u X_v = \beta (u,v) X_v X_u
\end{equation}
Notice that $\beta$ does not depend on the choice of the $X_t$ and that it is an invariant of the isomorphism class of the grading.
Now we check that $ \beta $ is an alternating bicharacter.
For that purpose, we define $ \sigma : T \times T \longrightarrow \mathbb{R}^{\times} $ such that $ X_u X_v = \sigma(u,v) X_{uv} $.
On the one hand:
\[ X_u X_v X_w = \sigma(u,v) X_{uv} X_w = \sigma(u,v) \beta(uv,w) X_w X_{uv} = \beta(uv,w) X_w X_u X_v \]
But on the other hand:
\[ X_u X_v X_w = \beta(v,w) X_u X_w X_v = \beta(u,w) \beta(v,w) X_w X_u X_v \]
Therefore $\beta$ is multiplicative in the first variable, and likewise in the second; so $ \beta $ is an alternating bicharacter.
Moreover, it has type I or II depending on the center of $\mathcal{D}$ (remember that, if $ \mathcal{D} \cong M_n (\mathbb{C}) $, then $iI$ is homogeneous).
Therefore Proposition \ref{prop-basis-alt-bich} gives us the condition on the group $T$ and, by dimension count, on $n$.

We define $ \mu : T_2 \longrightarrow \{ \pm 1 \} $ by the equation:
\begin{equation}\label{eq-quad-form}
X_t^2 = \mu (t) I
\end{equation}
Remark that $\mu$ is independent of the choice of the $X_t$, and that it is another invariant of the isomorphism class of the grading.
Note that $T_2$ is $T$ itself except if $ T \cong \mathbb{Z}_2^{2m-1} \times \mathbb{Z}_4 $.
We claim that $\mu$ is a quadratic form whose associated alternating bicharacter is the restriction of $\beta$ to $ T_2 \times T_2 $.
Indeed, for all $ u,v \in T_2 $, $ (X_u X_v)^2 = X_u X_v X_u X_v = \beta(u,v) \mu(u) \mu(v) I $.
So $ X_u X_v = \lambda X_{uv} $ with $ \lambda \in \{ +1 , -1 \} $; and $ (X_u X_v)^2 = (\lambda X_{uv})^2 = \mu(uv) I $.
Therefore $ \beta(u,v) = \mu(uv) \mu(u)^{-1} \mu(v)^{-1} $, and we have proved the claim.
Also remark that, if $ \mathcal{D} \cong M_n (\mathbb{C}) $ and we denote by $ f = \deg (iI) $ the semineutral element, then $\mu(f)=-1$.

As the homogeneous components have dimension $1$, the isomorphism class of the grading is completely determined by \eqref{eq-alt-bich} and \eqref{eq-quad-form}.
Conversely, we are going to find a graded division algebra for a given datum $(T,\mu)$ or $(T,\mu,\beta)$.

In order to do this, we need to recall Example \ref{exam-H-M2R-dim1}.
For the gradings in Equations \eqref{eq-H-dim1} and \eqref{eq-M2R-dim1} we get the same alternating bicharacter $ \beta : \mathbb{Z}_2^2 \times \mathbb{Z}_2^2 \longrightarrow \{ \pm 1 \} $, given by (denoting by $a$, $b$, $c$ the nontrivial elements of $\mathbb{Z}_2^2$):
\[ \beta(a,b) = \beta(b,a) = \beta(a,c) = \beta(c,a) = \beta(b,c) = \beta(c,b) = -1 \]
And all the other values equal to $1$.
This is the unique nontrivial alternating bicharacter on $\mathbb{Z}_2^2$.
However, for $\mathbb{H}$ the quadratic form $ \mu : \mathbb{Z}_2^2 \longrightarrow \{ \pm 1 \} $ is:
\[ \mu(e) = +1 \qquad \mu(a) = \mu(b) = \mu(c) = -1 \]
While for $ M_2 ( \mathbb{R} ) $ is:
\[ \mu(e) = \mu(a) = \mu(b) = +1 \qquad \mu(c) = -1 \]
Permuting $a$, $b$, $c$ we get two more quadratic forms of type I; and a simple computation tells us that these four are all the possible quadratic forms on $\mathbb{Z}_2^2$ such that its associated alternating bicharacters are nontrivial.
Remark also that, for the grading in Equation \eqref{eq-M2C-dim1}, the quadratic form $ \mu : \langle a,b^2 \rangle \longrightarrow \{ \pm 1 \} $ is:
\[ \mu(e) = \mu(a) = +1 \qquad \mu(b^2) = \mu(ab^2) = -1 \]

Let $\mu$ be a quadratic form of type I (on $ T \cong (\mathbb{Z}_2^2)^m $) and $\beta$ its associated alternating bicharacter.
By Proposition \ref{prop-basis-alt-bich}, there exists a symplectic basis $ \{ a_1 , \allowbreak b_1 , \ldots , a_m , \allowbreak b_m \} $.
The restriction of $\mu$ to $ \langle a_i,b_i \rangle \cong \mathbb{Z}_2^2 $ is one of the four quadratic forms on $\mathbb{Z}_2^2$ with nontrivial associated alternating bicharacter.
As $\mu$ is determined by $\beta$ and its values on $ a_1 , b_1 , \ldots , a_m , b_m $, we can choose as the graded algebra (after reordering the indexes $i$):
\begin{equation}\label{eq-MnR-dec}
M_2(\mathbb{R}) \otimes \ldots \otimes M_2(\mathbb{R}) \otimes \mathbb{H} \otimes \ldots \otimes \mathbb{H}
\end{equation}
Where each copy of $ M_2(\mathbb{R}) $ is graded as in Equation \eqref{eq-M2R-dim1} and each of $\mathbb{H}$ as in Equation \eqref{eq-H-dim1}.
Note that \eqref{eq-MnR-dec} is isomorphic to $M_n(\mathbb{R})$ (respectively $M_{n/2}(\mathbb{H})$) if and only if the number of copies of $\mathbb{H}$ is even (respectively odd).
So we have a different proof that the Arf invariant is well defined, because the algebras $M_n(\mathbb{R})$ and $M_{n/2}(\mathbb{H})$ are not isomorphic.

For type II we repeat the same reasoning, so we indicate only the new arguments.
If $ T \cong (\mathbb{Z}_2^2)^m \times \mathbb{Z}_2 $, the basis of Proposition \ref{prop-basis-alt-bich} will be $ \{ a_1 , \allowbreak b_1 , \ldots , a_m , \allowbreak b_m , \allowbreak f \} $.
But now, changing if necessary $a_i$ or $b_i$ by $ a_i f $ or $ b_i f $, we can get $ \mu (a_i) = \mu (b_i) = +1 $; and we can choose as the graded algebra:
\[ M_2(\mathbb{R}) \otimes \ldots \otimes M_2(\mathbb{R}) \otimes \mathbb{C} \]
Where each copy of $ M_2(\mathbb{R}) $ is graded as in Equation \eqref{eq-M2R-dim1} and $\mathbb{C}$ as in Equation \eqref{eq-C-dim1}.

While if $ T \cong (\mathbb{Z}_2^2)^{m-1} \times ( \mathbb{Z}_2 \times \mathbb{Z}_4 ) $, there are $ a_1 , b_1 , \ldots , a_m , b_m $ as in Proposition \ref{prop-basis-alt-bich}.
The only difference with the previous cases is the subgroup $ \langle a_m,b_m \rangle \cong \mathbb{Z}_2 \times \mathbb{Z}_4 $.
We know that $ \mu(f) = -1 $ (where $ f = b_m^2 $).
Changing if necessary $a_m$ by $a_m f$, we get $\mu(a_m)=1$ (and $ \mu(a_m f) = -1 $).
Therefore we can choose:
\[ M_2(\mathbb{R}) \otimes \ldots \otimes M_2(\mathbb{R}) \otimes M_2(\mathbb{C}) \]
Where each copy of $ M_2(\mathbb{R}) $ is graded as in Equation \eqref{eq-M2R-dim1} and $M_2(\mathbb{C})$ as in Equation \eqref{eq-M2C-dim1}.

Note that we have proved the assertion about the equivalence classes for $ \mathcal{D} \cong M_n ( \mathbb{C} ) $.
For $ \mathcal{D} \cong M_n ( \mathbb{R} ) $ or $ M_n ( \mathbb{H} ) $ it is enough to show that all the gradings as in Equation \eqref{eq-MnR-dec} with the same parity of copies of $\mathbb{H}$ are equivalent.
Consider the symplectic basis $ \{ a_1, b_1, a_2, b_2 \} $ of $ \mathbb{Z}_2^2 \times \mathbb{Z}_2^2 $ relative to the grading of $ \mathbb{H} \otimes \mathbb{H} $, that is, the quadratic form is determined by $ \mu(a_1) = \mu(b_1) = \mu(a_2) = \mu(b_2) = -1 $.
Then $ a_1' = a_1 a_2 $, $ b_1' = a_1 a_2 b_1 $, $ a_2' = b_1 b_2 $, $ b_2' = b_1 b_2 a_2 $ form another symplectic basis, but $ \mu(a_1') = \mu(b_1') = \mu(a_2') = \mu(b_2') = +1 $.
Therefore we can write $ \mathbb{H} \otimes \mathbb{H} $ as $ M_2(\mathbb{R}) \otimes M_2(\mathbb{R}) $ if we rename the elements of the group, so they are equivalent.
\end{proof}

\begin{remark}\label{rem-cliff}
Given an alternating bilinear map $ \beta : V \times V \longrightarrow \mathbb{F} $ on a finite-dimensional vector space $V$ over a field $\mathbb{F}$ of characteristic $2$ of maximal rank (so $\beta$ is nondegenerate if $\dim(V)$ is even and its rank is $\dim(V)-1$ if $\dim(V)$ is odd), $\beta$ is unique up to isometry, and this implies that there is a basis $ \{ e_1 , \ldots , e_r \} $ of $V$ such that $ \beta(e_i,e_j) = 1 $ for any $ i \neq j $.

Now, given a fine grading on an algebra $\mathcal{D}$ as in Theorem \ref{thm-MnR-dim1} or Theorem \ref{thm-MnC-dim1}.(1), this shows that there are elements $ e_1 , \ldots , e_r $ that generate the support of the grading with $ \beta(e_i,e_j) = -1 $ for any $ i \neq j $, where now $\beta$ is the associated alternating bicharacter of the quadratic form $\mu$ attached to the grading (the notation is multiplicative now).
Take $ X_i \in \mathcal{D}_{e_i} $ with $ X_i^2 = \mu(e_i) I $ ($ \mu(e_i) \in \{ \pm 1 \} $), for $ i = 1 , \ldots , r $.
Then $ X_i^2 = \pm I $ for any $i$, $ X_i X_j = - X_j X_i $ for any $ i \neq j $, and the elements $ X_1 , \ldots , X_r $ generate $\mathcal{D}$.

We conclude that $\mathcal{D}$ is graded isomorphic to the Clifford algebra of a nondegenerate quadratic form of signature equal to $ \vert \{ i : 1 \leq i \leq r , \, \mu(e_i) = +1 \} \vert $, on an $r$-dimensional real vector space, where the grading on the Clifford algebra is induced by a grading on the vector space (see \cite[Proposition 2.2]{AM}).
\end{remark}

\section{Coarsenings}\label{sec-coar}

As mentioned in the Introduction, we have two tools to reduce the study of division gradings with homogeneous components of dimension $2$ or $4$ to the case of dimension $1$ (see Remark \ref{rem-dim-hom-comp}).
The first is that our gradings come from coarsenings.
The second is the Double Centralizer Theorem.

\begin{remark}\label{rem-complex-grad}
If the real algebra $\mathcal{D}$ is isomorphic to $M_n(\mathbb{C})$ and it is graded so that its neutral component $\mathcal{D}_e$ is isomorphic to $\mathbb{C}$, then it can happen that $\mathcal{D}_e$ coincides with the center of the algebra.
In that case, the grading can be regarded as a grading of the complex algebra $M_n(\mathbb{C})$.
The classification of division gradings on $M_n(\mathbb{C})$ over the complex field is in \cite[Theorem 2.15]{EK}.
The isomorphism and equivalence classes that gives this classification remain the same over the real field, because the invariants that differentiate them are also preserved by isomorphisms of real algebras.
In the sequel we will study the remaining case $ \mathcal{D}_e \neq Z(\mathcal{D}) $.
\end{remark}

Consider a finite-dimensional simple real algebra $\mathcal{D}$ graded by an abelian group so that its neutral component $\mathcal{D}_e$ is isomorphic to $\mathbb{H}$, which is central simple.
Then, the Double Centralizer Theorem (see for instance \cite[Theorem 4.7]{Jacobson}) tells us that $ C_{\mathcal{D}}(\mathcal{D}_e) $ is also a simple algebra and that there is a natural isomorphism of algebras $ \mathcal{D} \cong \mathcal{D}_e \otimes C_{\mathcal{D}}(\mathcal{D}_e) $.
As $\mathcal{D}_e$ is (trivially) graded, so is its centralizer; hence the previous isomorphism is an isomorphism of graded algebras.
Therefore the classification of gradings on $ M_n (\mathbb{R}) $, $ M_n (\mathbb{C}) $ and $ M_n (\mathbb{H}) $ making the algebra a graded division algebra with homogeneous components of dimension $4$ reduces to that of dimension $1$ of $ M_{n/4} (\mathbb{H}) $, $ M_{n/2} (\mathbb{C}) $ and $ M_n (\mathbb{R}) $, respectively.
We summarize this argument as follows:

\begin{theorem}\label{thm-dim4}
Let $G$ be an abelian group and $\mathcal{D}$ a finite-dimensional simple real algebra.
Any division $G$-grading on $\mathcal{D}$ with homogeneous components of dimension $4$ is the product grading of the trivial grading on the neutral homogeneous component $\mathcal{D}_e$ and a division $G$-grading with homogeneous components of dimension $1$ on the centralizer $ C_{\mathcal{D}}(\mathcal{D}_e) $.
\end{theorem}

Let us finally turn our attention to division gradings with homogeneous components of dimension two.

\begin{proposition}\label{prop-ref}
Let $T$ be a finite abelian group, and let $\mathcal{D}$ be a finite-di\-men\-sio\-nal simple real algebra.
Consider a grading on $\mathcal{D}$ with support $T$ making $\mathcal{D}$ a graded division algebra so that the dimension of its homogeneous components is not $1$ and that the neutral component does not coincide with the center of $\mathcal{D}$.
Then, there exists a proper refinement of the grading.
\end{proposition}

\begin{proof}
Our idea is to find a homogeneous element $X$ such that $ X^2 \in Z(\mathcal{D}) $ and which does not commute with $\mathcal{D}_e$.
Then we can consider the inner automorphism $ \tau : d \in \mathcal{D} \longrightarrow XdX^{-1} \in \mathcal{D} $; which is involutive, since $ X^2 \in Z(\mathcal{D}) $, and satisfies $ \tau (\mathcal{D}_t) \subseteq \mathcal{D}_t $ for all $ t \in T $, because $X$ is homogeneous.
Therefore we can split the homogeneous components as direct sum of subspaces: $ \mathcal{D}_t = \mathcal{D}_{(\bar{0},t)} \oplus \mathcal{D}_{(\bar{1},t)} $, where $ \mathcal{D}_{(\bar{0},t)} := \{ d \in \mathcal{D}_t \mid \tau(d)=d \} $ and $ \mathcal{D}_{(\bar{1},t)} := \{ d \in \mathcal{D}_t \mid \tau(d)=-d \} $.
Clearly this defines a refinement of the grading by the group $ \mathbb{Z}_2 \times T $.
The refinement is proper, since $X$ does not commute with $\mathcal{D}_e$, and it remains a division grading, so necessarily all the homogeneous components have the same dimension and the support is the whole $ \mathbb{Z}_2 \times T $.

If the division algebra $\mathcal{D}_e$ is isomorphic to $\mathbb{H}$, we can take as $X$ any nonzero element of trace zero of $\mathcal{D}_e$.
So assume that $ \mathcal{D}_e = \mathbb{R} I \oplus \mathbb{R} J \cong \mathbb{C} $, where $I$ is the unity of $\mathcal{D}$ and $J^2=-I$.

Let $X$ be a homogeneous element of degree $t$ which does not commute with $J$.
There exists $ \lambda \in \mathcal{D}_e $ such that
$ X J = \lambda J X $, therefore $ J^2 X = X J^2 = \lambda J X J = \lambda J \lambda J X = \lambda^2 J^2 X $; so $\lambda=-1$ and $ X J = - J X $.
We can write $ X^{ \vert t \vert } = \alpha I + \beta J $, with $ \alpha , \beta \in \mathbb{R} $; and if $ \beta \neq 0 $, then $ J = \beta^{-1}( X^{ \vert t \vert } - \alpha I ) $, which commutes with $X$, a contradiction; thus $ X^{ \vert t \vert } \in Z(\mathcal{D}) $.
If $ \vert t \vert $ is odd, then $ X^{ \vert t \vert } J = - J X^{ \vert t \vert } $, another contradiction, so $ \vert t \vert $ is even.

Let $Y$ be a homogeneous element, we are going to prove that $Y$ commutes with $X^2$, so $ X^2 \in Z(\mathcal{D}) $.
We can write $ XY = (aI+bJ) YX $, with $ a,b \in \mathbb{R} $.
Therefore $ X^2 Y = X (aI+bJ) YX = (aI-bJ) XYX = (aI-bJ)(aI+bJ) YX^2 = c^2 YX^2 $, where $ c^2 := a^2+b^2 $.
Thus $ Y X^{ \vert t \vert } = X^{ \vert t \vert } Y = X^2 \cdots X^2 Y = c^{ \vert t \vert } Y X^{ \vert t \vert } $; which means that $c^{ \vert t \vert }=1$, and so $c^2=1$, that is, $ X^2 Y = YX^2 $.
\end{proof}

\begin{remark}\label{rem-ref}
The computations of the previous proof will be useful in the proof of Theorems \ref{thm-MnR-dim2} and \ref{thm-MnC-dim2}.
So we record for future reference that, if $ \mathcal{D}_e = \mathbb{R} I \oplus \mathbb{R} J \cong \mathbb{C} $ (where $I$ is the unity of $\mathcal{D}$ and $J^2=-I$), $ \mathcal{D}_e \neq Z(\mathcal{D}) $, $K$ is the support of the centralizer of the neutral component ($ K := \mathrm{supp} ( C_{\mathcal{D}} (\mathcal{D}_e) ) $), and $X,X'$ are two nonzero homogeneous elements of degree $ g \in T \setminus K $; then $XJ=-JX$ and:
\begin{equation}\label{eq-rem-ref}
X^2 \in Z(\mathcal{D}) \qquad (X')^2 \in \mathbb{R}_{>0} X^2
\end{equation}
For the last assertion, note that $ X' = (aI+bJ) X $ for some $ a,b \in \mathbb{R} $.

Moreover, $X$ defines a refinement of the grading with support $ \langle h \rangle \times T $ (where $h$ has order $2$) and such that the homogeneous elements of degree in $ \{ e \} \times T $ commute with $X$, while those of degree in $ \{ h \} \times T $ anticommute with $X$.
\end{remark}

\begin{theorem}\label{thm-MnR-dim2}
Let $G$ be an abelian group, and $\mathcal{D}$ a real algebra isomorphic to $M_n(\mathbb{R})$ (respectively $M_n(\mathbb{H})$).
If there is a division $G$-grading $\Gamma$ on $\mathcal{D}$ with homogeneous components of dimension $2$, then $ n = 2^m $ for some integer $ m \geq 1 $ (respectively $ m \geq 0 $) and the support of $\Gamma$ is a subgroup $T$ of $G$ isomorphic to $\mathbb{Z}_2^{2m-1}$ (respectively $\mathbb{Z}_2^{2m+1}$).
The isomorphism classes of such gradings are in one-to-one correspondence with the triples $(T,K,\nu)$ formed by a subgroup $T$ of $G$ isomorphic to $\mathbb{Z}_2^{2m-1}$ (respectively $\mathbb{Z}_2^{2m+1}$), an index $2$ subgroup $K$ of $T$, and a positive (respectively negative) nice map $\nu$ on $ T \setminus K $ of type I.
All such gradings belong to one equivalence class, represented by 2.(a) (respectively 2.(b)) in the list of Section \ref{sec-equiv}.
\end{theorem}

\begin{theorem}\label{thm-MnC-dim2}
Let $G$ be an abelian group, and $\mathcal{D}$ a real algebra isomorphic to $M_n(\mathbb{C})$.
If there is a division $G$-grading $\Gamma$ on $\mathcal{D}$ with homogeneous components of dimension $2$ \emph{such that the neutral component does not coincide with the center of $\mathcal{D}$}, then $ n = 2^m $ (for some integer $ m \geq 1 $) and the support of $\Gamma$ is a subgroup $T$ of $G$ isomorphic either to $\mathbb{Z}_2^{2m}$ or to $ \mathbb{Z}_2^{2m-2} \times \mathbb{Z}_4 $.
Moreover, the following possibilities appear:
\begin{enumerate}
	\item If the support is isomorphic to $\mathbb{Z}_2^{2m}$, the isomorphism classes of such gradings are in one-to-one correspondence with the triples $(T,K,\nu)$ formed by a subgroup $T$ of $G$ isomorphic to $\mathbb{Z}_2^{2m}$, an index $2$ subgroup $K$ of $T$, and a nice map $\nu$ on $ T \setminus K $ of type II.
	All such gradings belong to one equivalence class, represented by 2.(c) in the list of Section \ref{sec-equiv}.
	\item If the support is isomorphic to $ \mathbb{Z}_2^{2m-2} \times \mathbb{Z}_4 $, the isomorphism classes of such gradings are in one-to-one correspondence with the union of the following two sets (for such a subgroup $ T \cong \mathbb{Z}_2^{2m-2} \times \mathbb{Z}_4 $, denote by $f$ the generator of $ T^2 = \{ t^2 \mid t \in T \} $, and by $T_2$ the subgroup $ \{ t \in T \mid t^2=e \} $, as in Remark \ref{rem-basis-alt-bich}):
	\begin{enumerate}
		\item[(i)] The set of $4$-tuples $(T, \allowbreak K, \allowbreak \beta, \allowbreak \nu)$ formed by a subgroup $T$ of $G$ isomorphic to $ \mathbb{Z}_2^{2m-2} \times \mathbb{Z}_4 $, an index $2$ subgroup $K$ of $T$ different from $T_2$, an alternating bicharacter $\beta$ on $K$ of type II, and a nice map $\nu$ on $ T_2 \setminus ( K \cap T_2 ) $; such that the restriction of $\beta$ to $ ( K \cap T_2 ) \times ( K \cap T_2 ) $ is the associated alternating bicharacter of $\nu$, and that the quadratic forms induced by $\nu$ take the value $-1$ on $f$.
		\item[(ii)] The set of pairs $(T,[\nu])$ for\-med by a subgroup $T$ of $G$ isomorphic to $ \mathbb{Z}_2^{2m-2} \times \mathbb{Z}_4 $, and an equivalence class $[\nu]$ of the quotient set of nice maps $ \nu : T \setminus T_2 \longrightarrow \{ \pm 1 \} $ of type II with semineutral element $f$, with the equivalence relation $ \nu \sim \nu' $ if either $ \nu' = \nu $ or $ \nu' = - \nu $.
	\end{enumerate}
	All the gradings corresponding to 2.(i) (respectively 2.(ii)) are equivalent.
	But the gradings corresponding to 2.(i) are not equivalent to those corresponding to 2.(ii).
	The former are represented by 2.(d) in the list of Section \ref{sec-equiv}, and the latter by 2.(e).
\end{enumerate}
\end{theorem}

\begin{proof}[Proof of Theorems \ref{thm-MnR-dim2} and \ref{thm-MnC-dim2}]
Write $ \mathcal{D}_e = \mathbb{R} I \oplus \mathbb{R} J $ ($ \cong \mathbb{C} $), where $I$ is the unity and $J^2=-I$.
We start with the case $ \mathcal{D} \cong M_n(\mathbb{R}) $.
As the grading comes from a coarsening, the support $T$ is isomorphic to $ \mathbb{Z}_2 \times \mathbb{Z}_2^{2m-2} $.
Call $K$ the support of the centralizer of the neutral component:
\[ K := \mathrm{supp} ( C_{\mathcal{D}} (\mathcal{D}_e) ) \]
It is an invariant of the isomorphism class of the grading.
Equation \eqref{eq-rem-ref} gives us the other invariant, $ \nu : T \setminus K \longrightarrow \{ \pm 1 \} $, as the map defined by the equation:
\[ Y^2 = \nu(t) I \]
For all $ t \in T \setminus K $, and $ Y \in \mathcal{D}_t $ ``of norm $1$'', that is, such that $ Y^4 = I $.

Fix $ g \in T \setminus K$.
There exists an $ X \in \mathcal{D}_g $ such that $ X J = - J X $ and $ X^2 = \nu(g) I $ (see Remark \ref{rem-ref}), so the graded subalgebra $ \mathcal{D}_e \oplus \mathcal{D}_g $ is isomorphic to $M_2(\mathbb{R})$ or $\mathbb{H}$, graded as in Example \ref{exam-H-M2R-M2C-dim2}; in particular it is central simple.
We can apply the Double Centralizer Theorem to write $\mathcal{D}$ in the form:
\[ M_n(\mathbb{R}) \cong \mathcal{D} \cong ( \mathcal{D}_e \oplus \mathcal{D}_g ) \otimes C_{\mathcal{D}} ( \mathcal{D}_e \oplus \mathcal{D}_g ) \cong \left\{
	\begin{array}{l}
		M_2(\mathbb{R}) \otimes M_{n/2}(\mathbb{R}) \\
		\mathbb{H} \otimes M_{n/4}(\mathbb{H})
	\end{array}
\right. \]
Notice that the (graded division) subalgebra $ C_{\mathcal{D}} ( \mathcal{D}_e \oplus \mathcal{D}_g ) $ is isomorphic to $M_{n/2}(\mathbb{R})$ or $M_{n/4}(\mathbb{H})$ because it is simple.
In the first case $\nu(g)=+1$, while $\nu(g)=-1$ in the second case.
The support of $ C_{\mathcal{D}} ( \mathcal{D}_e \oplus \mathcal{D}_g ) $ is an index $2$ subgroup of $T$, and it is contained in $K$, so it is $K$.
If $ \mu : K \longrightarrow \{ \pm 1 \} $ is the quadratic form associated to the isomorphism class of $ C_{\mathcal{D}} ( \mathcal{D}_e \oplus \mathcal{D}_g ) $,
then $ \nu (gk) = \nu(g) \mu (k) $ for all $ k \in K $.
Therefore $\nu$ is a positive nice map of type I, because the Arf invariant of $\mu$ is $+1$ in the first case and $-1$ in the second case, and $\nu$ determines the isomorphism class of $ C_{\mathcal{D}} ( \mathcal{D}_e \oplus \mathcal{D}_g ) $, so it also determines the isomorphism class of $\mathcal{D}$.

Reciprocally, this decomposition provides a way to find a graded division algebra for a given $ \nu : T \setminus K \longrightarrow \{ \pm 1 \} $.
Indeed, if for example $ \nu (g) = +1 $, we can take $ M_2(\mathbb{R}) \otimes M_{n/2}(\mathbb{R}) $, where $M_2(\mathbb{R})$ is graded as in Equation \eqref{eq-M2R-dim2} with support $ \langle g \rangle $, and $M_{n/2}(\mathbb{R})$ has the $K$-division grading of the associated quadratic form $\mu_g$ of $\nu$ in $g$.
Finally, all such gradings are equivalent, because we can always find a $ g' \in T \setminus K $ such that $ \nu(g') = +1 $, and we know that all the $\mathbb{Z}_2^{2m-2}$-division gradings on $M_{n/2}(\mathbb{R})$ belong to one equivalence class.

The arguments for $ \mathcal{D} \cong M_n(\mathbb{H}) $ are exactly the same, but now the representatives of the isomorphism classes have the form ($ T \cong \mathbb{Z}_2 \times \mathbb{Z}_2^{2m} $):
\[ M_n(\mathbb{H}) \cong \left\{
	\begin{array}{l}
		M_2(\mathbb{R}) \otimes M_{n/2}(\mathbb{H}) \\
		\mathbb{H} \otimes M_n(\mathbb{R})
	\end{array}
\right. \]
For $ \mathcal{D} \cong M_n(\mathbb{C}) $ and $ T \cong \mathbb{Z}_2 \times \mathbb{Z}_2^{2m-1} $ the same reasoning gives:
\[ M_n(\mathbb{C}) \cong \left\{
	\begin{array}{l}
		M_2(\mathbb{R}) \otimes M_{n/2}(\mathbb{C}) \\
		\mathbb{H} \otimes M_{n/2}(\mathbb{C})
	\end{array}
\right. \]

The last case is $ \mathcal{D} \cong M_n(\mathbb{C}) $ and $ T \cong \mathbb{Z}_2 \times ( \mathbb{Z}_2^{2m-3} \times \mathbb{Z}_4 ) $.
If the support $ K = \mathrm{supp} ( C_{\mathcal{D}} ( \mathcal{D}_e ) ) $ is different from $ T_2 $, we can fix a $ g \in T \setminus K $ of order $2$, and argue as before to decompose $ \mathcal{D} $ in the form:
\[ M_n(\mathbb{C}) \cong \left\{
	\begin{array}{l}
		M_2(\mathbb{R}) \otimes M_{n/2}(\mathbb{C}) \\
		\mathbb{H} \otimes M_{n/2}(\mathbb{C})
	\end{array}
\right. \]
The only difference with the previous cases is that we can only define $\nu$ on $ T_2 \setminus ( K \cap T_2 ) $.
To determine the isomorphism class of $ C_{\mathcal{D}} ( \mathcal{D}_e \oplus \mathcal{D}_g ) $ we need also its associated alternating bicharacter, $ \beta : K \times K \longrightarrow \{ \pm 1 \} $ (see Theorem \ref{thm-MnC-dim1}.(2)).
$\beta$ is an invariant because it comes from the commutation relations of the $\mathcal{D}_e$-algebra $ C_{\mathcal{D}} (\mathcal{D}_e) $, that is, it is defined by the following equation:
\[ X_u X_v = \beta (u,v) X_v X_u \]
For all $ u,v \in K $, $ X_u \in \mathcal{D}_u $ and $ X_v \in \mathcal{D}_v $.

Finally we are going to prove \ref{thm-MnC-dim2}.2.(ii).
Assume that $ \mathcal{D} \cong M_n(\mathbb{C}) $, $ T \cong \mathbb{Z}_4 \times \mathbb{Z}_2^{2m-2} $ and $ \mathrm{supp} ( C_{\mathcal{D}} ( \mathcal{D}_e ) ) = T_2 $.
Fix a $ g \in T \setminus T_2 $, and a subgroup $A$ of $T_2$ such that $ T_2 \cong \langle f \rangle \times A $, which is equivalent to $ T \cong \langle g \rangle \times A $.
Note that $g^2=f$, and that $f$ is the degree of $iI$, because the grading comes from a coarsening (see Remark \ref{rem-basis-alt-bich}, Lemma \ref{lem-iI} and Theorem \ref{thm-MnC-dim1}).
For any $ t \in T \setminus T_2 $ and $ 0 \neq Y \in \mathcal{D}_t $, $ Y^2 \in Z(\mathcal{D})_f = \mathbb{R}(iI) $, so $ Y^4 \in \mathbb{R}_{>0} I $, and we can normalize to get $ Y^4 = -I $.
By Equation \eqref{eq-rem-ref}, we can define $ \nu : T \setminus T_2 \longrightarrow \{ \pm 1 \} $ as the unique map that satisfies the equation:
\[ Y^2 = \nu(t) X^2 \]
For all $ t \in T \setminus T_2 $, and normalized elements $ X \in \mathcal{D}_g $ and $ Y \in \mathcal{D}_t $, that is, $ X^4 = Y^4 = -I $.
Note that if we had fixed a different $g$, we could have $-\nu$ instead of $\nu$, but $[\nu]$ is an invariant of the isomorphism class of $\mathcal{D}$.
We will denote by $\mu_g$ the associated map of $\nu$ in $g$ (see Lemma \ref{lem-nice}).

We will prove later that we can decompose $\mathcal{D}$ as:
\begin{equation}\label{eq-MnC-dim2-2b-dec}
M_n(\mathbb{C}) \cong \mathcal{D} \cong C_{\mathcal{D}} (\mathcal{F}) \otimes \mathcal{F} \cong \left\{
	\begin{array}{l}
		M_2(\mathbb{C}) \otimes M_{n/2}(\mathbb{R}) \\
		M_2(\mathbb{C}) \otimes M_{n/4}(\mathbb{H})
	\end{array}
\right.
\end{equation}
Where $\mathcal{F}$ ($ \cong M_{n/2}(\mathbb{R}) $ or $M_{n/4}(\mathbb{H})$) is a graded division subalgebra with support $A$, while $C_{\mathcal{D}}(\mathcal{F})$ ($ \cong M_2(\mathbb{C}) $) is graded as in Equation \eqref{eq-M2C-dim2} with support $ \langle g \rangle $.
If $ \mu : A \longrightarrow \{ \pm 1 \} $ is the quadratic form (of type I) associated to $\mathcal{F}$, then $  \nu(ga) = \nu(g) \mu(a) $ for all $ a \in A $; therefore $\mu$ is the restriction of $\mu_g$ to $A$ and $[\nu]$ determines the isomorphism class of $\mathcal{D}$.

Now we are going to show that $\nu$ is a nice map of type II with semineutral element $f$.
We can consider the element $iI$, which has degree $f$, to obtain:
\begin{equation}\label{eq-MnC-dim2-2b-nice}
\text{For all } a \in A \text{, } \nu(gfa) = -\nu(ga) \text{, and so } \mu_g(fa) = -\mu_g(a) \text{.}
\end{equation}
Define $ \beta : T_2 \times T_2 \longrightarrow \{ \pm 1 \} $ by $ \beta(u,v) := \mu_g(uv) \mu_g(u)^{-1} \mu_g(v)^{-1} $ (for all $ u,v \in T_2 $).
We know that the restriction of $\beta$ to $ A \times A $ is an alternating bicharacter of type I, and applying \eqref{eq-MnC-dim2-2b-nice} we check that $ \beta(fa_1,fa_2) = \beta(fa_1,a_2) = \beta(a_1,fa_2) = \beta(a_1,a_2) $ for all $ a_1,a_2 \in A $.
Therefore $\beta$ is an alternating bicharacter of type II with semineutral element $f$.
Finally we note that $ \mu_g(f) = -1 $, because of \eqref{eq-MnC-dim2-2b-nice}.

Reciprocally, there is a graded division algebra for a given $[\nu]$:
Up to isomorphism, it is either $ M_2(\mathbb{C}) \otimes M_{n/2}(\mathbb{R}) $ or $ M_2(\mathbb{C}) \otimes M_{n/4}(\mathbb{H}) $, where $M_2(\mathbb{C})$ is graded as in Equation \eqref{eq-M2C-dim2} with support $ \langle g \rangle $, and the grading of $ M_{n/2}(\mathbb{R}) $ or $ M_{n/4}(\mathbb{H}) $ is defined by the restriction to $A$ of the associated quadratic form of $\nu$ in $g$ (the associated quadratic forms of $\nu$ and $-\nu$ coincide).
Indeed, if $\nu'$ is the nice map associated to this algebra, then \eqref{eq-MnC-dim2-2b-nice} implies that $ [\nu'] = [\nu] $.

Now we will prove \eqref{eq-MnC-dim2-2b-dec}.
Consider a refinement of the grading induced by a homogeneous element $X$ of degree $g$, with support $ T' := \langle h \rangle \times \langle g \rangle \times A $, as in Remark \ref{rem-ref}.
Let $(\beta',\mu')$ be the associated pair (see Theorem \ref{thm-MnC-dim1}), $ T'_2 := \langle h \rangle \times \langle f \rangle \times A $, and $\mathcal{F}$ the graded subalgebra of support $ A \leq T' $ (that is, the direct sum of the homogeneous components of the refinement of degree in $A$), which is also a graded subalgebra relative to our original grading.

As the homogeneous elements that commute with $J$ are those whose degree has order $1$ or $2$, and $J$ has degree $h$ in this new grading, $ \langle h,f \rangle \subseteq \mathrm{rad} ( \beta' \vert_{ T'_2 \times T'_2 } ) $.
But Remark \ref{rem-basis-alt-bich} tells us that this radical has exactly $4$ elements, so $ \langle h,f \rangle = \mathrm{rad} ( \beta' \vert_{ T'_2 \times T'_2 } ) $.
Therefore the restriction of $\mu'$ to $ A \leq T'_2 $ has type I.
This restriction determines $\mathcal{F}$, so $ \mathcal{F} \cong M_{n/2}(\mathbb{R}) $ or $ M_{n/4}(\mathbb{H}) $, and in particular $\mathcal{F}$ is central simple.
Therefore $ \mathcal{D} \cong C_{\mathcal{D}} (\mathcal{F}) \otimes \mathcal{F} $, by the Double Centralizer Theorem.

Moreover, the homogeneous elements of degree in $ \{ e \} \times \langle g \rangle \times A \leq T' $ commute with $X$, so $ X \in C_{\mathcal{D}}(\mathcal{F}) $; while those of degree in $ \langle h \rangle \times \langle f \rangle \times A \leq T' $ commute with $J$, so also $ J \in C_{\mathcal{D}}(\mathcal{F}) $.
We deduce that the support of $ C_{\mathcal{D}}(\mathcal{F}) $ is $ \langle g \rangle \leq T $ and that $ C_{\mathcal{D}}(\mathcal{F}) $ is isomorphic to $M_2(\mathbb{C})$ graded as in Equation \eqref{eq-M2C-dim2}, because $XJ=-JX$ and $X^4=-I$.
This finishes the proof of \eqref{eq-MnC-dim2-2b-dec}.

To conclude, we analyze the equivalence classes.
We must show that two graded algebras of the form \eqref{eq-MnC-dim2-2b-dec} are equivalent.
This is clear except if one is of the form $ M_2(\mathbb{C}) \otimes M_{n/2}(\mathbb{R}) $ with support $T^1$ and the other is of the form $ M_2(\mathbb{C}) \otimes M_{n/4}(\mathbb{H}) $ with support $T^2$.
In that case we consider, with the previous notation, the refinements of supports $ (T^1)' = \langle h_1 \rangle \times T^1 = \langle h_1 \rangle \times \langle g_1 \rangle \times A_1 $ and $ (T^2)' = \langle h_2 \rangle \times T^2 = \langle h_2 \rangle \times \langle g_2 \rangle \times A_2 $.
By Theorem \ref{thm-MnC-dim1}, we know that these refinements are equivalent, that is, there exist an isomorphism of algebras $ \varphi : M_2(\mathbb{C}) \otimes M_{n/2}(\mathbb{R}) \longrightarrow M_2(\mathbb{C}) \otimes M_{n/4}(\mathbb{H}) $ and a group isomorphism $ \alpha : \langle h_1 \rangle \times \langle g_1 \rangle \times A_1 \longrightarrow \langle h_2 \rangle \times \langle g_2 \rangle \times A_2 $ such that $\varphi$ sends homogeneous elements of degree $t$ to homogeneous elements of degree $\alpha(t)$.
But, as we noted in Remark \ref{rem-basis-alt-bich}, $h_1$ and $h_2$ are distinguished elements, so necessarily $ \alpha(h_1) = h_2 $ and $\alpha$ induces an isomorphism $ T^1 = (T^1)' / \langle h_1 \rangle \longrightarrow T^2 = (T^2)' / \langle h_2 \rangle $.
Therefore, we can take as our desired equivalence the same $\varphi$.
Finally, the gradings of \ref{thm-MnC-dim2}.2.(i) are not equivalent to those of \ref{thm-MnC-dim2}.2.(ii), since only in the first case the support of the centralizer of the neutral component has elements of order 4.
\end{proof}

\section*{Acknowledgements and license}

I want to thank the advice of my thesis supervisor, Alberto Elduque; especially for his ideas, that have been instrumental in the development of this paper.

Supported by a doctoral grant of the Diputaci\'on General de Arag\'on.
Also supported by the Spanish Ministerio de Econom\'ia y Competitividad and the Fondo Europeo de Desarrollo Regional FEDER (MTM2013-45588-C3-2-P); and by the Diputaci\'on General de Arag\'on and the Fondo Social Europeo (Grupo de Investigaci\'on de \'Algebra).

This work is licensed under the Creative Commons Attribution-NonCommercial-NoDerivatives 4.0 International License. To view a copy of this license, visit http://creativecommons.org/licenses/by-nc-nd/4.0/ or send a letter to Creative Commons, PO Box 1866, Mountain View, CA 94042, USA.

\end{document}